\documentclass[10pt]{amsart}
\usepackage{amssymb}
\usepackage{hyperref}

\newtheorem{theo}{Theorem} 
\newtheorem{lemma}{Lemma}[section]
\newtheorem{prop}[lemma]{Proposition}
\newtheorem{corol}[lemma]{Corollary}

\newtheorem{reftheo}[lemma]{Theorem} 

\theoremstyle{remark}
\newtheorem{remark}[lemma]{\bf{Remark}}

\theoremstyle{definition}

\newcommand{\CC}{\mathbb{C}}
\newcommand{\NN}{\mathbb{N}}
\newcommand{\RR}{\mathbb{R}}

\newcommand{\eps}{\varepsilon}

\newcommand{\ubar}{\overline{u}}

\newcommand{\brx}{\langle x \rangle}

\DeclareMathOperator{\re}{Re}
\DeclareMathOperator{\im}{Im}

\newcommand{\loc}{\rm loc}
\newcommand{\hdot}{\dot{H}^1}

\DeclareMathOperator{\supp}{supp}

\DeclareMathOperator{\divergence}{div}

\def\d{{\partial}}

\title[Strichartz estimates on manifolds]{Weighted Strichartz estimates 
for radial Schr\"odinger equation on noncompact manifolds}
\author[V.~Banica]{}
\email{Valeria.Banica@univ-evry.fr}
\address{Valeria Banica\\
Universit\'e d'Evry Val d'Essonne\\
D\'epartement de Math\'ematiques\\
Bd. F. Mitterrand, 91025 Evry Cedex\\ 
France}
\author[T.~Duyckaerts]{}
\email{thomas.duyckaerts@u-cergy.fr}
\address{Thomas Duyckaerts\\
Universit{\'e} de Cergy-Pontoise\\
D\'epartement de Math\'ematiques\\ 
Site de Saint Martin, 2 avenue Adolphe-Chauvin\\ 
95302 Cergy-Pontoise cedex, Frances}

\begin{document}

\subjclass[2000]{}
\maketitle
\begin{center}{{\bf{V. Banica$^1$, T. Duyckaerts$^{2}$}}  \vspace{3mm}\\\tiny{$^1$D\'epartement de Math\'ematiques, Universit\'e d'Evry, France\\
$^2$ D\'epartement de Math\'ematiques, UMR CNRS 8088, Universit\'e de Cergy-Pontoise, France}}
\end{center}
\begin{abstract}
We prove global weighted Strichartz estimates for radial solutions of linear Schr\"odinger equation on a class of rotationally symmetric noncompact manifolds, generalizing the known results on hyperbolic and Damek-Ricci spaces. This yields classical Strichartz estimates with a larger class of exponents than in the Euclidian case and improvements for the scattering theory.  The manifolds, whose volume element grows polynomially or exponentially at infinity, are characterized essentially by negativity conditions on the curvature. In particular the rich algebraic structure of hyperbolic and Damek-Ricci spaces is not the cause of the improved dispersive properties of the equation. The proofs are based on known dispersive results for the equation with potential on the Euclidean space, and on a new one, valid for $C^1$ potentials decaying like $1/r^2$ at infinity. 
\end{abstract}

\tableofcontents

\section{Introduction}
Let us consider the linear Schr\"odinger equation on a $n$-dimensional Riemannian manifold $(M,g)$
\begin{equation}
  \label{eq:ls}
\left\{\begin{array}{c}  
i\d_t u + \Delta_{M} u=f,\\
u(0)=u_0\in L^2(M).
\end{array}\right.
\end{equation}
where $\Delta_M$ is the associated Laplace-Beltrami operator. In the Euclidian case $(M,g)=(\RR^n,\delta)$, the solutions of \eqref{eq:ls} satisfy the Strichartz Estimates (see \cite{St77a,GiVe85,Ya87,CaWe88,KeTa98}):
\begin{equation}
\label{StrichartzIntro}
\|u\|_{L^{p_1}(\RR,L^{q_1}(M))}\leq C\left(\left\|u_0\right\|_{L^2(M)}+\|f\|_{L^{p_2}(\RR,L^{q_2}(M))}\right),
\end{equation}
where $(p_1,q_1)$ and $(p_2,q_2)$ are any n-admissible couples,
\begin{equation}
\label{admissible}
\frac{2}{p_j}+\frac{n}{q_j}=\frac{n}{2},\quad p_j\geq 2, \quad (p_j,q_j,n)\neq(2,\infty,2).
\end{equation}
The validity of \eqref{StrichartzIntro}, or weaker related estimates, on other manifolds that Euclidean space has been intensively studied the last twenty years (\cite{Bo93}, \cite{BuGeTz04AJ},\cite{BuGeTz02} \cite{StTa02}, \cite{HaTaWu04}, \cite{RoZu05BO}, \cite{An06}, \cite{BlSmSo06P} \cite{BoTz06P}, \cite{BoPr}, etc). To our knowledge, the only cases in which an improvement to \eqref{StrichartzIntro} is known are  hyperbolic space and the much larger class of Damek-Ricci spaces. For such manifolds, radial solutions of \eqref{eq:ls} satisfy Strichartz estimates with a weight in space related to the growth of the volume density, which constitutes a gain at infinity (\cite{Ba05}, \cite{Pi05}, \cite{BaCaSt06}). Moreover, global Strichartz estimates hold for a larger class of Lebesgue exponents, which implies, as shown in \cite{BaCaSt06}, an improvement for the scattering theory of the nonlinear equation. This behavior is expected to hold even in the nonradial case (\cite{AnPiPr}). \par
The Damek-Ricci spaces are examples of noncompact harmonic spaces which are not necessarily symmetric, yielding counterexamples in the noncompact case to Lichnerowicz's conjecture (\cite{DaRi92}). By harmonic we mean that the volume density is radial at any point. These spaces have been constructed algebraically from generalized Heisenberg groups. They have nonpositive sectional curvature valued in $[-1,0]$, and negative constant Ricci curvature (see \cite{BeTrVa95BO}).
The improvement of the dispersive properties of the Schr\"odinger equation on these spaces is usually explained by the negative curvature and the exponential growth of the volume element. However they are Lie Groups, with a large group of isometries, and one might think for example that this rich algebraic structure is also necessary to get the improved dipersive properties. \par
The aim of this paper is to give other examples of noncompact manifolds for which there are gains as in the Damek-Ricci spaces.\par
There are examples of manifolds having only some of the properties of the Damek-Ricci spaces, and that do not present major dispersive improvements in the radial case. The Euclidean space $\mathbb{R}^n$ is harmonic with a rich algebraic structure and with zero curvature. In the radial setting, there are some improvements, but not very strong, in the sense that they only yield the 2-d endpoint Strichartz estimate (\cite{Ta00},\cite{Vi01}, see also Remark \ref{euclidradial}). On Heisenberg groups, noncompact Lie groups which are not harmonic and whose sectional curvature takes positive and negative values, the local Strichartz inequalities do not hold (\cite{BaGeXu00}).

In the compact setting, the Strichartz estimates may also be related to the sign of the curvature of the manifold. On the flat torus local Strichartz estimates hold with an arbitrary small loss \cite{Bo93}. On the spheres, harmonic manifolds with positive curvature, the local Strichartz inequalities hold only with important loss of derivatives, and the result is sharp (\cite{BuGeTz02}). Note that in these examples, it is meaningless to look for a gain at infinity, and that the lack of global in time estimates is an immediate consequence of the compactness of the manifold. However the fact that the estimates are better in the case of the torus may be related to the strict positivity of the curvature of the sphere. It remains to our knowledge an open question in which way local Strichartz estimates hold on compact manifolds with constant negative curvature. Such a manifold is obtained as quotient of the hyperbolic space by a discrete co-compact subgroup of its isometry group.

In the noncompact case, it seems reasonable to think that negativity conditions on the curvature are sufficient to get the improved global Strichartz estimates. In this work we show that this holds for rotationally symmetric manifolds in the radial case under the additional assumption that the volume density grows polynomially at infinity. We get similar results for manifolds whose volume density grows exponentially at infinity. This also yields improvements for the nonlinear scattering theory, in the spirit of \cite{BaCaSt06}. Concerning the local in time estimates, boundeness conditions on the sectional curvature are sufficient for obtaining weighted estimates. These results show in particular that the algebraic structure is not necessary to get the weighted Strichartz estimates.

We will call $n-$dimensional (noncompact) \textbf{rotationally symmetric manifold} a manifold $M$ given by the metric
$$ds^2=dr^2+\phi^2(r)\,d\omega^2,$$
where $d\omega^2$ is the metric on the sphere $\mathbb{S}^{n-1}$, and  
$\phi$ is a $\mathcal{C}^\infty$ nonnegative function on $[0,\infty)$, strictly positive on $(0,\infty)$, such that $\phi^{(even)}(0)=0$ and $\phi'(0)=1$.
These conditions on $\phi$ ensures us that the manifold is smooth (\S 1.3.4. of \cite{Pe98BO}). The volume element is $\phi^{n-1}(r)$, and the Laplace-Beltrami operator on $M$ is
\begin{equation}\label{Laplace}
\Delta_M=\partial^2_r+(n-1)\frac{\phi'(r)}{\phi(r)}\partial_r+\frac{1}{\phi^2(r)}\Delta_{\mathbb{S}^{n-1}},
\end{equation}
Let us notice that under appropriate conditions on $\phi$, restricting ourselves to radial functions, the operator \eqref{Laplace} may also be viewed as the Laplace-Beltrami operator on a rotationally symmetric manifold of other dimension than $n$ (see Remark \ref{euclidradial}).

For such manifold, the curvature of $M$, can be computed explicitely in terms of $\phi$ (see \S 3.2.3 of \cite{Pe98BO}). Indeed, there exists an orthonormal frame $(F_j)_{j=1\ldots n}$ on $(M,g)$, where $F_n$ corresponds to the radial coordinate, and $F_1,\ldots,F_{n-1}$ to the spherical coordinates, for which $F_i\land F_j$ diagonalize the curvature operator $\mathcal{R}$ : 
$$\mathcal{R}(F_i\land F_n)=-\frac{\phi''}{\phi}F_i\land F_n\,\,\,,\,\,\, i<n,$$
$$\mathcal{R}(F_i\land F_j)=-\frac{(\phi')^2-1}{\phi^2}\,F_i\land F_j\,\,\,,\,\,\, i,j<n.$$
The Ricci curvature is then given by 
$$Ric(F_i)=-\left((n-2)\frac{(\phi')^2-1}{\phi^2}+\frac{\phi''}{\phi}\right)F_i\,\,\,,\,\,\, i<n\quad,\qquad Ric(F_n)=-(n-1)\frac{\phi''}{\phi}F_n,$$
and the scalar curvature is
$$scal=-2(n-1)\frac{\phi''}{\phi}-(n-1)(n-2)\frac{(\phi')^2-1}{\phi^2}.$$
We will focus on the sectional curvature $sec_r$, which is a normalized quadratic form on the tangent space $T_rM$, and takes in our case the following extremal values $sec_r^{rad}$ and $sec_r^{tan}$
\begin{equation}
\label{sec}
sec_r^{rad}=-\frac{\phi''}{\phi}\quad,\quad sec_r^{tan}=-\frac{(\phi')^2-1}{\phi^2}.
\end{equation}

We start with a simple result concerning the local Strichartz estimates with gain in space. 
\begin{prop}\label{prop.local}
Let $M$ be a rotationally symmetric manifolds of dimension $n\geq 3$ such that 
\begin{equation}\label{secbounded}
\exists\, m>0,\quad \frac{1}{\phi(r)}+\left|sec^{rad}_r\right|\leq m\quad\forall\,r\in[1,\infty)
\end{equation} 
Then for all $T>0$, there exists a constant $C$ such that for all radial solutions $u$, $f$ of \eqref{eq:ls}
\begin{equation}\label{wlocstrichartz}
\left\|u\,\left(\frac{\phi(r)}{r}\right)^{\frac{n-1}{2}\left(1-\frac{2}{q_1}\right)}\right\|_{L^{p_1}((0,T),L^{q_1}(M))}\leq C\left\|u_0\right\|_{L^2(M)}+C\left\|f\left(\frac{r}{\phi(r)}\right)^{\frac{n-1}{2}\left(1-\frac{2}{q_2}\right)}\right\|_{L^{p_2'}((0,T),L^{q_2'}(M))},
\end{equation}
where $(p_1,q_1)$ and $(p_2,q_2)$ are any $n-$admissible couples. Notice that if the volume density grows faster than in the Euclidean case, then $\frac{\phi}{r}$ is a gain in space. 
\end{prop}
Condition \eqref{secbounded} implies that the growth of $\phi$ is at most exponential at infinity. For larger growth it is not clear that the weighted Strichartz estimate still holds (see Remark \ref{quadra}).

\begin{remark}
The growth of the weight function $\frac{\phi(r)}{r}$ can be related to a sign condition on the curvature in the following way. Let us suppose that the tangential sectional curvatures is nonpositive starting from an $r_0$. Then $\phi'(r)\geq 1$ for all $r\geq r_0$, because $\phi'$ is a  continuous function and, as $\phi$ is positive,  we cannot have $\phi'(r)\leq -1$ for all $r\geq r_0$. It follows 
$$\phi(r)-\phi(r_1)=\int_{r_0}^r\phi'(s)\,ds\geq r-r_0,$$
so as $r$ goes to infinity
$$1\lesssim \frac{\phi(r)}{r}.$$

Proposition \ref{prop.local} does not yield any gain when the tangential sectional curvature is positive at infinity. Indeed in this case $\frac{\phi}{r}$ is bounded. 
Note that in the case of noncompact manifold, the curvature cannot be "too positive": if the manifold $M$ is complete and the sectional curvature bounded from below by a positive constant, then $M$ must be compact (see \cite[Theorem 4.1]{Pe98BO}).
\end{remark}

In Proposition \ref{prop.local} no non-trapping condition is imposed. This is due to the fact that we are working in the radial setting. In the non-radial case \eqref{wlocstrichartz} is not true in general for trapping manifolds. 
\\

We now turn to the global estimates. 

If $l\in \RR$, $k\in \NN^*$ and $\eps$ is a $C^k$ function for $r\geq 1$, we write 
$$\eps(r)=o_k(r^l),\; r\rightarrow+\infty$$
 if there exist a constant $C>0$ such that
$$ \forall j\in\{0,\ldots,k\},\; \forall r\geq 1,\quad |\eps^{(j)}(r)|\leq Cr^{l-j}.$$

We will state a result for manifolds such that the volume element grows polynomially at infinity. See Theorem \ref{theo.exp'} in Section \ref{sec:rev} for an analogue in the case of exponential grow.

\begin{theo}\label{theo.poly}
Let $M$ be a rotationally symmetric manifold of dimension $n\geq 3$ and let $m>\frac{1}{n-1}$. Assume that
\begin{gather}
\label{negativecurvature}
\exists \delta_0>0,\quad\forall r\geq 0,\quad sec^{rad}_r\leq \left(\frac{1}{2(n-1)}-\delta_0\right)\frac{1}{r^2},\\
\label{polynomialbehaviour}
\exists A>0,\quad \phi(r)=A r^m+o_3(r^m),\quad r\rightarrow +\infty.
\end{gather}
Then the radial solutions of the free equation \eqref{eq:ls} satisfy for all n-admissible couples $(p_j,q_j)$ the weighted Strichartz estimate
\begin{equation}\label{wstrichartz}
\left\|u\left(\frac{\phi(r)}{r}\right)^{\frac{n-1}{2}\left(1-\frac{2}{q_1}\right)}\right\|_{L^{p_1}(\mathbb{R},L^{q_1}(M))}\leq C\left\|u_0\right\|_{L^2(M)}+C\left\|f\left(\frac{r}{\phi(r)}\right)^{\frac{n-1}{2}\left(1-\frac{2}{q_2}\right)}\right\|_{L^{p_2}(\mathbb{R},L^{q_2}(M))}.
\end{equation}
Furthermore if $m>1$, and 
$$ N:=m(n-1)+1$$
then for any $d\in (n,N)$, the solutions of \eqref{eq:ls} satisfy all global $d$-admissible Strichartz estimates. 
\end{theo}
Notice that if $sec^{rad}_r$ is nonpositive, assumption \eqref{negativecurvature} holds.
\begin{remark}
\label{rem.euclidian}
Under the assumptions of the preceding theorem, if $m>1$ and $N$ is an integer, the volume element at infinity is $\phi^{n-1}dr\approx r^{m(n-1)}dr$, which is the volume element of $\RR^{N}$. In this case the radial solutions of \eqref{eq:ls} admit all Strichartz estimates without weight for couples that are between  $n$-admissible and $N$-admissible. Note if $d_1<d_2$, the $d_1$-admissible couples are better from the point of view of local well-posedness, whereas the $d_2$-admissible couples yield a better decay for large time, and thus stronger scattering results (see Corollary \ref{corol.scattering}).
\end{remark}
\begin{remark}
Theorem \ref{theo.poly} also gives weighted Strichartz estimates in the case $\frac{1}{n-1}<m<1$. In this case, the weight is a loss compared to the usual estimates. The assumption $\frac{1}{n-1}<m$ means that the volume density is larger, at infinity, that the one of the Euclidian plane $\RR^2$.
\end{remark}

It is easy to give examples of manifolds $M$ satisfying the assumptions of Theorem \ref{theo.poly}. For example, take
\begin{equation}
\label{defphipoly}
\phi(r)=r+a_1 r^3+...+a_kr^{2k+1},
\end{equation}
where $k\geq 1$ and $a_i>0$, $i=1\ldots k$.

\begin{remark}
It is also possible to get sufficient condition for the weighted Strichartz estimates in term of the square root of the volume element:
\begin{equation}
\label{deftau}
\tau=\phi^{\frac{n-1}{2}}.
\end{equation}
Namely, the conclusions of Theorem \ref{theo.poly} still hold if assumption \eqref{negativecurvature} is replaced by the assumption that their exists $\delta_0>0$ such that
\begin{equation}
\tag{\ref{negativecurvature}'}
\frac{\tau''}{\tau}\geq -\frac{1/4-\delta_0}{r^2}.
\end{equation}
We refer to Proposition \ref{prop.manif} for a general result. 
\end{remark}

\begin{remark}
The above results hold not only for rotationally symmetric manifolds, but for all manifold $M^n$ admitting a global coordinate system $(r,\theta)$ for which the radial part of the Laplacian equals to $\partial_r^2+(n-1)\frac{\phi'}{\phi}\partial_r$ and the volume element is $\phi^{n-1}dr$. Furthermore as a consequence of Theorem \ref{theo.dispersive} below, a local-in-space $1/2$-smoothing effect also holds, that we did not state for the sake of brievity.
\end{remark}

Let us turn to the consequence of the preceding result in term of nonlinear scattering.
We will say that the equation 
\begin{equation}
\label{NLS}
i\d_t u+\Delta_{M} u\pm|u|^{p}u=0,\quad u(0)=u_0\in H^1(M)
\end{equation}
has \textbf{short-range behaviour} when for all $u_0\in H^1$, there exists $\tilde{u}_0\in H^1$ such that
$$ \lim_{t\rightarrow +\infty} \|u(t)-e^{it\Delta}\tilde{u}_0\|_{H^1}=0.$$

\begin{corol}\label{corol.scattering}
Assume that the conditions of Theorem \ref{theo.poly} hold and that $m>1$. Let $p\in (0,+\infty)$ such that
$$\frac{4}{N}< p<\frac{4}{n-2},$$
where $N$ is defined in Theorem \ref{theo.poly}. Then \eqref{NLS} has short-range behavior.
\end{corol}
Let us recall that on the Euclidean space the critical power is $\frac{2}{n}$: for smaller powers the solutions cannot have the behavior of a free solution (\cite{St74BO}, \cite{Ba84}). The bound $\frac{4}{N}$ of Corollary \ref{corol.scattering} is better as soon as the power $m$ in \eqref{polynomialbehaviour} is larger than $2+\frac{1}{n-1}$. Note that the upper bound for scattering is still $\frac{4}{n-2}$, which is the $\RR^n$ upper bound.

The proofs of Proposition \ref{prop.local} and Theorem \ref{theo.poly} rely on a change of unknown function in \eqref{eq:ls} similar to the one of Pierfelice in \cite{Pi05}, and related to the volume density. In the radial case the equation is reduced to the linear Schr\"odinger equation 
\begin{equation}
\label{LS}
\tag{$S_V$}
\left\{\begin{array}{c}  
i\partial_t v+\Delta v -V v=g,\\
v(0)=v_0\in L^2(\RR^n).
\end{array}\right.
\end{equation}
with the particular potential
\begin{equation}\label{V}
V=\frac{\tau''}{\tau}-\frac{(n-1)(n-3)}{4r^2},\text{ where }\tau=\phi^{\frac{n-1}{2}}.
\end{equation}
The conditions $\phi'(0)=1$ and $\phi^{(even)}(0)=0$ imply that $V$ is bounded and smooth near $0$. The assumptions of Proposition \ref{prop.local}, ensure the boundedness of $V$ for large $r$, which is sufficient to get the local in time Strichartz estimates.

Under the assumptions of Theorem \ref{theo.poly}, the potential $V$ defined in \eqref{V} decays like $1/r^2$ at infinity, which is critical for global Strichartz estimates (see \cite{GoVeVi06} for a counterexample when the decay is slower). For potential of order $1/r^2$, under positivity and repulsion assumptions on $V$ (analoguous to our assumptions \eqref{HypPositive} and \eqref{HypRepulsive} below), Strichartz estimates are shown in \cite{BuPlStTZ04}. We also refer to \cite{BaRuVe06} for the smoothing effect and to \cite{RoSc04}, where dispersion is shown in dimension $3$, with a potential whose decay is almost critical at infinity. In this last work a lower local regularity as well as time-dependence are allowed (see also \cite{Goldberg06}). \par
The assumptions in \cite{BuPlStTZ04} are well-suited for a potential with at pole at the origin, but do not always cover our case. We give a variant of their results which is more adapted to potentials that are also smooth at the origin. More precisely, we consider the linear Schr\"odinger equation \eqref{LS} with real potential $V$ on $\RR^n$, $n\geq 3$ and define the following assumptions
\begin{gather}
\label{HypBound}
\tag{H1}
\exists C>0,\;\forall x\in \RR^n,\quad |V(x)|\leq \frac{C}{\brx^2},\\
\label{HypPositive}
\tag{H2}
\exists \delta_0>0,\;\left(\frac{n}{2}-1\right)^2+r^2 V\geq \delta_0,\\
\label{HypRepulsive}
\tag{H3}
\exists R>0,\quad |x|\geq R\Longrightarrow \left(\frac{n}{2}-1\right)^2-r^2\partial_r(r V)\geq \delta_0.
\end{gather}
where $\brx=(1+|x|^2)^{1/2}$, $r=|x|$ and $\partial_r$ is the radial derivative $\frac{x}{|x|}\cdot \partial_x$. 

Note that these assumptions are similar to the one in \cite{BuPlStTZ04}, expect that in their case the potential need not be bounded at the origin, and that the analogue of \eqref{HypRepulsive} must hold for any $x\neq 0$.

\begin{theo}
\label{theo.dispersive}
Assume that $n\geq 3$ and that $V\in C^{1}(\RR^n)$ satisfies assumptions \eqref{HypBound}, \eqref{HypPositive} and \eqref{HypRepulsive}. Then the (possibly nonradial) solutions of \eqref{LS} satisfy the following.\\ 
i) Smoothing effect : there exists $C>0$ such that for all $g$ with $\brx g\in L^2\big(\RR,H^{-1/2}\big)$ we have
\begin{equation}
\label{smoothing}
\|\brx^{-1} v\|_{L^2(\RR,H^{1/2})}\leq C\left(\|v_0\|_{L^2}+\|\brx g\|_{L^2(\RR,H^{-1/2})}\right).
\end{equation}
ii) Strichartz estimates : there exists $C>0$ such that for all n-admissible couples $(p_1,q_1)$, $(p_2,q_2)$, and  
for all solution of \eqref{LS} with $g\in L^{p_2'}(\RR,L^{q_2'})$ we have
\begin{equation}
\label{Strichartz}
\|v\|_{L^{p_1}(\RR,L^{q_1})}\leq C\left(\|v_0\|_{L^2}+\|g\|_{L^{p_2'}(\RR,L^{q_2'})}\right).
\end{equation}
\end{theo}
\begin{remark}
\label{rem.constant}
Theorem \ref{theo.dispersive} remains valid if for some real constant $\beta$, $V-\beta$ satisfies the assumptions \eqref{HypBound} \eqref{HypPositive} and \eqref{HypRepulsive}. Indeed $\tilde{v}=e^{i\beta t}v$ is solution of the equation \eqref{LS} with the potential $V-\beta$ instead of $V$. This yields global Strichartz estimates for manifold such that the volume element grows exponentially at infinity (see Theorem \ref{theo.exp'} below).
\end{remark}

Let us a give a quick idea of the proof of Theorem \ref{theo.dispersive}. Following the strategy of \cite{BuPlStTZ04}, \eqref{Strichartz} is deduced from \eqref{smoothing}. Estimate \eqref{smoothing} is the consequence of an uniform weighted estimate on the resolvent $(-\Delta+V-\lambda)^{-1}$, which is classical except near $\lambda=0$. To treat this last case, which is closely related to the lack of resonance at $0$ for the operator $-\Delta+V$, we use a resolvent estimate shown in \cite{BuPlStTZ04}.\par

We finish this introduction with a few remarks and related open problems. We first note that we can extend the above dispersive results to the radial wave equation (see Lemma \ref{waves}).
Furthermore, for the sake of simplicity, we wrote the results in terms of $C^{\infty}$ manifolds. However the proof shows that Proposition \ref{prop.local} still holds when $\phi$ is of class $C^2$, and Theorem \ref{theo.poly} when $\phi$ is of class $C^3$.

In the present work we consider only radial solutions of \eqref{eq:ls}, which do not see the trapped geodesics of the manifold $M$. In the general nonradial case, we expect that the preceding results should hold under a non-trapping condition on the metric. However our method does not seem to adapt easily in the nonradial setting, where a new term $\frac{1}{\phi^2}\Delta_{S^{n-1}}v$ appears in \eqref{LS}. The fact that Theorem \ref{theo.dispersive} also holds for nonradial potentials and solutions is not helpful here. We refer to \cite{AnPiPr} for some results in this direction.

The validity of weighted (or classical) Strichartz inequalities for rotationally symmetric manifolds such that $\phi$ has an growth which is intermediate between polynomial and exponential is to our knowledge still open. This problem is related to the study of \eqref{LS} with a radial positive potential $V$ whose decay is of order $\frac{1}{|x|^{s}}$, $1<s<2$ at infinity. When $V$ is homogeneous of degree $s$ for large $|x|$ and nonradial, global Strichartz estimates fail in general (see \cite{GoVeVi06}). It is also the case by an adaptation of the example in \cite{Duyckaerts07} if $V$ is radial, tends to $0$ a little slower than $\frac{1}{|x|^2}$, but does not satisfy any analogue of our assumption \eqref{HypRepulsive}. However, the question remains to our knowledge still open, even for radial solutions of \eqref{LS}, when $V$ is a radial smooth positive potential decaying slower than $\frac{1}{|x|^2}$, and satisfying a repulsion assumption at infinity, for example:
$$ \frac{C}{(1+r)^s}\geq |V(r)|,\quad V(r)\geq \frac{\delta_0}{(1+r)^s}, \quad -(rV)'\geq \frac{\delta_0}{(1+r)^s}, \quad \delta_0>0,\; 1<s<2.$$
It seems also a difficult question to know, when $V$ is exactly of order $\frac{1}{|x|^2}$ at infinity, if the assumption \eqref{HypRepulsive} is necessary. A positive result in this direction would allow us to get Strichartz estimates without any condition on the first derivative of the curvature at infinity. Such a condition, which is assumption \eqref{tau'''<0} in Proposition \ref{prop.manif} below, is hidden in assumption \eqref{polynomialbehaviour}.

The paper is organized as follows. 
In \S\ref{gen} we describe the tranformation of \eqref{eq:ls} on a general rotationally symmetric manifold $M$ into the linear Schr\"odinger equation with potential on the Euclidean space, and its consequences in terms of  dispersive estimates.
In \S\ref{local} we prove Proposition \ref{prop.local}, and in \S\ref{globalimpr} we prove Theorem \ref{theo.poly} and Theorem \ref{theo.exp'}, which is its analogue when the volume element grows exponentially at infinity.
 In \S\ref{pot} we prove Theorem \ref{theo.dispersive}, first showing resolvent estimates (\S\ref{sub.resolvent}), then infering smoothing (\S\ref{sub.smooth}) and finally Strichartz estimates (\S\ref{sub.Str}).

 \thanks{The authors would like to thank R\'emi Carles, Olivier Drouet, Sorin Dumitrescu and Andrei Iftimovici for valuable discussions.}
 
\section{Gains for NLS on manifolds}
\label{sec:rev}

\subsection{General approach}\label{gen}

We consider the linear Schr\"odinger equation \eqref{eq:ls}
$$\left\{\begin{array}{c}  
i\d_t u+\Delta_{M}u=f,\\
u(0)=u_0
\end{array}\right.
$$
where
$$\Delta_M=\partial^2_r+(n-1)\frac{\phi'(r)}{\phi(r)}\partial_r+\frac{1}{\phi^2(r)}\Delta_{\mathbb{S}^{n-1}},$$
Let
$$\sigma(r):=\left(\frac{r}{\phi(r)}\right)^\frac{n-1}{2},\quad u(t,r)=\sigma(r)v(t,r),\quad f(t,r)=\sigma(r)g(t,r).$$
Then $v$ satisfies
$$i\d_t v+\d_r^2 v+\left(2\frac{\sigma'}{\sigma}+(n-1)\frac{\phi'}{\phi}\right)\d_r v+\left(\frac{\sigma''}{\sigma}+(n-1)\frac{\phi'}{\phi}\frac{\sigma'}{\sigma}\right) v+\frac{1}{\phi^2}\Delta_{\mathbb{S}^{n-1}}v=g.$$
Note that $\log \sigma=\frac{n-1}{2}\left(\log r-\log \phi\right)$, So that
\begin{equation}
\label{difflog}
\frac{\sigma'}{\sigma}=\frac{n-1}{2}\left(\frac 1r-\frac{\phi'}{\phi}\right).
\end{equation}
Thus the radial derivative part is the one of the Laplacian on $\mathbb{R}^n$. Differentiating \eqref{difflog} we get
$$\frac{\sigma''}{\sigma}-\frac{\sigma'^2}{\sigma^2}=\frac{n-1}{2}\left(-\frac{1}{r^2}-\frac{\phi''}{\phi}+\frac{\phi'^2}{\phi^2}\right).$$
Thus 
\begin{equation*}
\frac{\sigma''}{\sigma}+(n-1)\frac{\phi'}{\phi}\frac{\sigma'}{\sigma}=\frac{n-1}{2}\left(-\frac{1}{r^2}-\frac{\phi''}{\phi}+\frac{\phi'^2}{\phi^2}+2\frac{\phi'}{\phi}\frac{\sigma'}{\sigma}\right)+\frac{\sigma'^2}{\sigma^2}.
\end{equation*}
Hence, using the expression of $\frac{\sigma'}{\sigma}$ given by \eqref{difflog}
$$\frac{\sigma''}{\sigma}+(n-1)\frac{\phi'}{\phi}\frac{\sigma'}{\sigma}=\frac{(n-1)(n-3)}{4}\frac{1}{r^2}-\frac{(n-1)(n-3)}{4}\left(\frac{\phi'}{\phi}\right)^2-\frac{n-1}{2}\,\frac{\phi''}{\phi}.$$
In conclusion $v$ satisfies
$$i\d_t v+\Delta_{\mathbb{R}^n} v+\left(\frac{1}{\phi^2(r)}-\frac{1}{r^2}\right)\Delta_{\mathbb{S}^{n-1}}v-V(r)v=\frac f\sigma,$$
with, as in \eqref{V},
\begin{equation}
\label{expressionV2}
V(r)=\frac{n-1}{2}\frac{\phi''}{\phi}+\frac{(n-1)(n-3)}{4}\,\left(\left(\frac{\phi'}{\phi}\right)^2-\frac{1}{r^2}\right)=\frac{\tau''}{\tau}-\frac{(n-1)(n-3)}{4}\frac{1}{r^2},
\end{equation}
If $v$ is a radial solution, then we are in the case of the Schr\"odinger equation on $\mathbb{R}^n$ with potential $V$
\begin{equation}\label{eqV}
\left\{\begin{array}{c}
i\d_t v+\Delta_{\mathbb{R}^n} v-V(r)v=g,\\
v(0)=v_0=\frac{u_0}{\sigma}.
\end{array}\right.
\end{equation}
Therefore we get easily the following lemma.
\begin{lemma}\label{locgen}
If $V(r)$ is a potential such that the radial solutions of \eqref{eqV} enjoys the Strichartz estimates on a time interval $I$, then the radial solutions of equation \eqref{eq:ls} satisfy the weighted Strichartz estimates
\begin{equation}\label{wlocgen}
\left\|u\,\sigma^{-\left(1-\frac{2}{q_1}\right)}\right\|_{L^{p_1}(I,L^{q_1}(M))}\leq c\left\|u_0\right\|_{L^2(M)}+\left\|f\,\sigma^{\left(1-\frac{2}{q_2}\right)}\right\|_{L^{p_2'}(I,L^{q_2'}(M))},
\end{equation}
for all n-admissible couples $(p_i,q_i)$. Moreover, if $d>n$ is such that    
\begin{equation}
\tag{I1} \int_0^\infty \sigma^{\frac{2d}{d-n}}\phi^{n-1}\,dr <+\infty,
\end{equation}
then $u$ satisfies to all $d$-admissible Strichartz estimates on $I$. 
\end{lemma}
\begin{proof}
We write the Strichartz estimates for $v$ in terms of $u$
$$\left\|\frac{u}{\sigma}\right\|_{L^{p_1}(I,L^{q_1}(\mathbb{R}^n))}\leq C\left\|\frac{u_0}{\sigma}\right\|_{L^2(\mathbb{R}^n)}+C\left\|\frac f\sigma\right\|_{L^{p_2'}(I,L^{q_2'}(\mathbb{R}^n))}.$$
Knowing that the volume element on $M$ is $\phi^{n-1}dr$, for a general function $h$,
$$\left\|\frac{h}{\sigma}\right\|_{L^q(\mathbb{R}^n)}^q=\int_0^\infty \left|\frac{h}{\sigma}\right|^qr^{n-1}dr=\int_0^\infty|h|^q\sigma^{2-q}\,\phi^{n-1}dr=
\left\|h\,\sigma^{\frac{2}{q}-1}\right\|_{L^q(M)}^q,$$
and therefore the first assertion of the lemma is proved. 
For ending with some d-admissible couples Strichartz estimates, $d>n$ we shall use the weight in space, as in \cite{BaCaSt06}.  We compute by H\"older's inequality the $d$-endpoint norm
$$\|u\|_{L^2(I,L^{\frac{2d}{d-2}}(M))}\leq C\left\|u\sigma^{-\frac{1}{n}}\right\|_{L^2(I,L^{\frac{2n}{n-2}}(M))}\left\|\sigma^{\frac{1}{n}}\right\|_{L^{\frac{nd}{d-n}}(M)}.$$
The integrability condition $(I1)$ gives us
$$\|u\|_{L^2(I,L^{\frac{2d}{d-2}}(M))}\leq C\left\|u\,\sigma^{-\frac 1n}\right\|_{L^2(I,L^{\frac{2n}{n-2}}(M))}.$$
The weighted estimates \eqref{wlocgen} implies on one hand
\begin{equation}
\label{Strichartzi}
\|u\|_{L^2(I,L^{\frac{2d}{d-2}}(M))}\leq C\left\|u_0\right\|_{L^2(M)}+C\left\|f\right\|_{L^1(I,L^2(M))},
\end{equation}
and on other hand
$$\|u\|_{L^2(I,L^{\frac{2d}{d-2}}(M))}\leq C\left\|u_0\right\|_{L^2(M)}+C\left\|f\,\sigma^\frac{1}{n}\right\|_{L^2(I,L^{\frac{2n}{n+2}}(M))}.$$
By applying H\"older's inequality in the last estimate
$$\|u\|_{L^2(I,L^{\frac{2d}{d-2}}(M))}\leq C\left\|u_0\right\|_{L^2(M)}+C\left\|f\right\|_{L^2(I,L^{\frac{2d}{d+2}}(M))}\left\|\sigma^\frac{1}{n}\right\|_{L^{\frac{nd}{d-n}}(M)},$$
so by using again the integrability condition $(I1)$ we get
\begin{equation}
\label{Strichartzibis}
\|u\|_{L^2(I,L^{\frac{2d}{d-2}}(M))}\leq C\left\|u_0\right\|_{L^2(M)}+C\left\|f\right\|_{L^2(I,L^{\frac{2d}{d+2}}(M))}.
\end{equation}

Now, multiplying \eqref{eq:ls} by $\overline{u}$ and taking the imaginary part we get by H\"older estimate
$$ \|u\|_{L^{\infty}(I,L^2(M))}^2\leq \|u_0\|_{L^2(M)}^2+\|u\|_{L^p(I,L^q(M))}\|f\|_{L^{p'}(I,L^{q'}(M))}.$$
Applying this to $p=\infty$, $q=2$ and $p=2$, $q=\frac{2d}{d-2}$, and using, in this last case, inequality \eqref{Strichartzibis}, we get
\begin{align}
\label{masscons}
\|u\|_{L^{\infty}(I,L^2(M))}&\leq C\|u_0\|_{L^2(M)}+C\|f\|_{L^1(I,L^{2}(M))}\\
\label{Strichartzii}
\|u\|_{L^{\infty}(I,L^2(M))}&\leq C\|u_0\|_{L^2(M)}+C\|f\|_{L^{2}(I,L^{\frac{2d}{d+2}}(M))}.
\end{align}

All $d$-admissible estimates follows now by interpolation from \eqref{Strichartzi}, \eqref{Strichartzibis}, \eqref{masscons} and \eqref{Strichartzii}.
\end{proof}
When the estimates are global in time, we get the following lemma, used already in \cite{BaCaSt06}, which will yield Corollary \ref{corol.scattering}.
\begin{lemma}\label{glgen}Let $d>n$.
We suppose that the global Strichartz estimates without weights hold for $d$-admissible couples. Then short-range wave operators exist for the radial nonlinear equation \eqref{NLS} for all powers $\frac{4}{d}\leq p<\frac{4}{d-2}$. 
\end{lemma}
\begin{proof}
If $d$ is an integer, one can obtain the short-range wave operators on $\mathbb{R}^d$ for 
$$i\d_t u+\Delta_{\mathbb{R}^d} u\pm|u|^{p}u=0,$$
for all powers $\frac{4}{d}\leq p<\frac{4}{d-2}$, just by using the global $d$-admissible Strichartz estimates, H\"older and Sobolev embeddings (\cite[\S 5]{Gi97BO}). In our case we have all these ingredients (even when $d$ is not an integer), and the lemma follows.
\end{proof}
We end this subsection with a lemma on the wave equation.
\begin{lemma}\label{waves}
We suppose that $\frac{\phi'}{\phi}$ is bounded at infinity. 
If $V(r)$ is a potential such that the radial solutions of the associated wave equation on $\mathbb{R}^n$ enjoy the Strichartz estimates, then the radial solutions of
$$\left\{\begin{array}{l} 
\d_t^2 u-\Delta_{M}u=f,\\
\d_tu(0)=u_1,\;u(0)=u_0
\end{array}\right.
$$
satisfy the weighted Strichartz estimates
\begin{equation}\label{w}
\left\|u\,\sigma^{-\frac 12\left(1-\frac{2}{q_1}\right)}\right\|_{L^{p_1}(I,L^{q_1}(M))}\leq C\left\|u_0\right\|_{H^1(M)}+C\left\|u_1\right\|_{L^2(M)}+C\left\|f\,\sigma^{\frac{1}{2}\left(1-\frac{2}{q_1}\right)}\right\|_{L^{p_2'}(I,L^{q_2'}(M))},
\end{equation}
for all wave n-admissible couples $(p_i,q_i)$.
\end{lemma}
\begin{proof}
The proof follows as the one of Lemma \ref{locgen}, the only difference being estimating the homogeneous $H^1$ norm of $u_0/\sigma$ on $\mathbb{R}^n$ in terms of the $H^1$ norm of $u_0$ on $M$. We have
$$\int_0^\infty \left|\nabla\left(\frac {u_0}{\sigma}\right)\right|^2r^{n-1}dr=\int_0^\infty \left|\frac{\d_ru_0}{\sigma}-\frac{u_0\sigma'}{\sigma^2}\right|^2r^{n-1}dr=\int_0^\infty \left|\d_ru_0-u_0\frac{\sigma'}{\sigma}\right|^2\phi^{n-1}dr=$$
$$=\int_0^\infty \left|\d_ru_0-u_02(n-1)\left(\frac 1r-\frac{\phi'}{\phi}\right)\right|^2\phi^{n-1}dr\leq \|u_0\|_{H^1(M)}+c\left\|u_0\left(\frac 1r-\frac{\phi'}{\phi}\right)\right\|_{L^2(M)}.$$
We get the boundeness of $\frac 1r-\frac{\phi'}{\phi}$ by the conditions on $\phi$ at zero, 
$$\frac 1r-\frac{\phi'}{\phi}=\frac{\phi-r\phi'}{r\phi}=\frac{r+o(r^2)-r(1+o(r))}{r(r+o(r^2))},$$
by the positivity of $\phi$ outside zero, and by the boundeness of $\frac{\phi'}{\phi}$ at infinity.
\end{proof}
\begin{remark}
We have used dispersive results on the wave equation on $\RR^n$ with potential to get informations about the free wave equation on manifolds. This was already done the other way around in \cite{Ta01}, where the author first derives weighted estimates for the wave equation on the hyperbolic space $\mathbb{H}^n$ (which is related to the symbol of the wave operator), and then, by a change of functions, gets estimates for the wave equation on $\mathbb{R}^n$(see also \cite{Ge00BO}). 
\end{remark}

\subsection{Local in time improvements}\label{local}
Let us prove Proposition \ref{prop.local}. In view of Lemma \ref{locgen}, to get Proposition \ref{prop.local} we need to show local in time Strichartz estimates for \eqref{eqV}. We will show that $V$ is bounded. By assumption \eqref{secbounded}, $\frac{\phi''}{\phi}$ and $\frac{1}{\phi}$ are  bounded by $m$. Let us check that $\frac{\phi'}{\phi}$ is bounded for $r\geq 1$. We have
$$\frac{d}{dr}\left(\frac{\phi'}{\phi}\right)=\frac{\phi''}{\phi}-\left(\frac{\phi'}{\phi}\right)^2\leq m-\left(\frac{\phi'}{\phi}\right)^2.$$
Thus if for some $r_1>0$, $\frac{\phi'(r_1)}{\phi(r_1)}\leq-\sqrt{m}$, then $\frac{\phi'(r)}{\phi(r)}\leq -\sqrt{m}$ for all $r\geq r_1$. Hence
$$ \forall r\geq r_1,\quad \phi(r)\leq \phi(r_1)e^{-\sqrt{m}(r-r_1)},$$
contradicting the fact that $\frac{1}{\phi}$ is bounded for $r\geq 1$. Thus $\frac{\phi'}{\phi}$ is bounded from below for $r\geq 1$. 
 
To show that $\frac{\phi'}{\phi}$ is bounded from above for large $r$, write
$$ \frac{d}{dr}\left(e^{\sqrt{m}r}\phi'-\sqrt{m}e^{\sqrt{m}r}\phi\right)=e^{\sqrt{m}r}\phi''-me^{\sqrt{m}r}\phi\leq 0.$$
As $\phi(0)=0$ and $\phi'(0)=1$, we get $e^{\sqrt{m}r}\phi'(r)\leq \sqrt{m}e^{\sqrt{m}r}\phi(r)+1$. From the fact that $\frac{1}{\phi}$ is bounded for $r\geq 1$, we obtain that $\frac{\phi'}{\phi}$ is bounded from above. As a conclusion $|V|\leq C\left[\frac{1}{r^2}+\frac{|\phi''|}{\phi}+\left(\frac{\phi'}{\phi}\right)^2\right]$ is bounded for $r\geq 1$.

By \eqref{expressionV2}, near $r=0$, it is sufficient to show that  $\frac{(\phi')^2}{\phi^2}-\frac{1}{r^2}$ and $\frac{\phi''}{\phi}$ are bounded.  
Since $\phi'(0)=1$ and $\phi^{(even)}(0)=0$, we have at zero:
\begin{equation}
\label{DL1}
\frac{\phi''}{\phi}=\frac{\phi'''(0)r+o(r^2)}{r+o(r)}=\phi'''(0)+o(r).
\end{equation}
and
\begin{align}
\notag
\frac{(\phi')^2}{\phi^2}-\frac{1}{r^2}&=\frac{\left(1+\frac{\phi'''(0)}{2}r^2+o(r^3)\right)^2}{r^2\left(1+\frac{\phi'''(0)}{6}r^2+o(r^3)\right)^2}-\frac{1}{r^2}\\
\label{DL2}
\frac{(\phi')^2}{\phi^2}-\frac{1}{r^2}&=\frac{1}{r^2}\left(1+\phi'''(0)r^2+o(r^3)\right)\left(1-\frac{\phi'''(0)}{3}r^2+o(r^3)\right)-\frac{1}{r^2}=\frac{2}{3}\phi'''(0)+o(r),
\end{align}
so we get boundeness for $r\leq 1$. Therefore the potential $V$ is bounded.

It is classical and easy to check that for such potential, local in time Strichartz estimates hold (see e.g. \cite[Theorem 1.1]{DAPiVi05}), which concludes the proof of Proposition \ref{prop.local}.
\qed

\begin{remark}\label{quadra}
Let us notice that \eqref{wlocstrichartz} still holds for all $M$ such that $V$ is subquadratic, with additional assumptions on the derivatives of $V$ : this is a consequence of the local dispersion proved by Fujiwara in \cite{Fu79} (see also \cite{Ca06} for the linear growth level). In the case when the potential is super-quadratic, local in time Strichartz estimates are only known with loss of derivative (\cite{YaZh04}, \cite{RoZu06}). There are simple examples of functions $\phi$ giving such potentials, for instance if $n=3$, $\phi(r)=r$ near the origin and $\phi(r)=e^{r^k}$ at infinity, for some $k>1$. This  yields smooth rotationally symmetric manifolds with negative sectional curvature, and with a volume density increasing very fast at infinity. In this case the potential 
\begin{equation}\label{expV}
V=\frac{\phi''}{\phi}=k(k-1)r^{k-2}+k^2r^{k-1}
\end{equation}
is growing at infinity, subquadratic if $k\in(1,2]$ and super-quadratic if $k>2$. In view of the results of \cite{YaZh04} and \cite{RoZu06}, we do not expect the local in time weighted Strichartz estimates \eqref{wlocstrichartz} to hold.\par
Notice that if $V(r)$ tend to infinity as $r$ tends to infinity, the operator $-\Delta+V$ has eigenvalues (see \S3.3 of \cite{BeShBO}), so that the global Strichartz estimates cannot hold for \eqref{LS}, and implicitely \eqref{wlocstrichartz} cannot hold globally in time on the corresponding manifold. Also, in this case the operator $-\Delta_M$ has eigenvalues, thus one cannot expect global Strichartz estimates without a bound on $V$ at infinity.
\end{remark}

\subsection{Global in time improvements}\label{globalimpr}
Assuming Theorem \ref{theo.dispersive}, we will show from Theorem \ref{theo.dispersive} the following general result. We will then prove Theorems \ref{theo.poly}, and state and prove Theorem \ref{theo.exp'}.
\begin{prop}\label{prop.manif}
Let $M$ be a rotationally symmetric manifold of dimension $n\geq 3$. Assume that there exists $c_0\in\RR$, $\delta_0>0$ such that  
\begin{gather}
\label{boundtau''}
\left|\frac{\tau''}{\tau}-c_0\right|\leq \frac{C}{r^2},\quad \forall r>0,\\
\label{tau''>0}
r^2\left(\frac{\tau''}{\tau}-c_0\right)\geq \delta_0-\frac{1}{4},\quad \forall r>0, \\
\label{tau'''<0}
\exists R>0,\; r\geq R\Longrightarrow -r^2\frac{d}{dr}\left(r\left(\frac{\tau''}{\tau}-c_0\right)\right)\geq \delta_0-\frac{1}{4},
\end{gather}
Then the radial solutions of the free equation \eqref{eq:ls} satisfy for all n-admissible couples $(p_j,q_j)$ the weighted Strichartz estimate \eqref{wstrichartz}.
\end{prop}
Note that \eqref{boundtau''} implies that $\tau$ grows polynomially (when $c_0=0$) or exponentially (when $c_0>0$) at infinity, which explains our restrictions on the growth of $\phi$ in Theorems \ref{theo.poly} and \ref{theo.exp'}. In our applications to manifolds, we ignored the case $c_0<0$, which would impose an exponential decay of $\phi$ at infinity.
\begin{proof}[Proof of Proposition \ref{prop.manif}]
In view of Lemma \ref{locgen} and Remark \ref{rem.constant}, it is enough to show that the potential $V-c_0$ satisfies the assumptions of Theorem \ref{theo.dispersive}, where $V$ is the potential defined in \eqref{expressionV2}. Since $\phi$ is $\mathcal{C}^\infty$ and positive for $r\geq 0$, $V$ is $C^{\infty}$ outside $0$. 
By \eqref{DL1} and \eqref{DL2}, noting that the $o(r)$ in these developments are also $C^1$ functions, $\frac{\phi''}{\phi}$ and $\frac{(\phi')^2}{\phi^2}-\frac{1}{r^2}$ are $C^1$ near $0$.
Thus $V$ is $C^1$ on $\RR^n$.

Conditions  \eqref{boundtau''}-\eqref{tau'''<0} are exactly \eqref{HypBound}-\eqref{HypRepulsive} for $V-c_0$. Thus all assumptions of Theorem \ref{theo.dispersive} hold, which shows that the Strichartz estimates hold for \eqref{LS}. Proposition \ref{prop.manif} follows from Lemma \ref{locgen}.
\end{proof}

We now turn to the proof of the global estimates in the polynomial case.
\begin{proof}[Proof of Theorem \ref{theo.poly}]
In this case, we will use Proposition \ref{prop.manif} with $c_0=0$. 
We first check \eqref{tau''>0}. We have
$$\frac{\tau''}{\tau}=\frac{n-1}{2}\frac{\phi''}{\phi}+\frac{(n-1)(n-3)}{4}\,\left(\frac{\phi'}{\phi}\right)^2.$$
By assumption \eqref{negativecurvature}, $\frac{\phi''}{\phi}\geq \left(\delta_0-\frac{1}{2(n-1)}\right)\frac{1}{r^2}$, which gives \eqref{tau''>0}.

By assumption \eqref{polynomialbehaviour}, $\phi(r)=Ar^m+o_3\left(r^m\right),\;r\rightarrow +\infty$, where $A>0$ and the notation $o_l\left(r^k\right)$ is defined just before Theorem \ref{theo.poly}. Thus
$$ \tau=\phi^{\frac{n-1}{2}}=\left(Ar^m+o_3\left(r^m\right)\right)^{\frac{n-1}{2}}=A^{\frac{n-1}{2}}r^{\frac{m(n-1)}{2}}\left(1+o_3\left(r^0\right)\right)^{\frac{n-1}{2}},\quad r\rightarrow +\infty.$$
By the formula $(1+u)^{\frac{n-1}{2}}=1+u\int_0^{1}\frac{n-1}{2}(1+us)^{\frac{n-3}{2}}ds$, we get that 
$\left(1+o_3\left(r^0\right)\right)^{\frac{n-1}{2}}=1+o_3(r^0)$. Let $N=m(n-1)+1$. We have
$$ \tau=A^{\frac{n-1}{2}}r^{\frac{N-1}{2}}+o_3\left(r^{\frac{N-1}{2}}\right),\quad \tau''=A^{\frac{n-1}{2}}\frac{(N-1)(N-3)}{4}r^{\frac{N-5}{2}}+o_1\left(r^{\frac{N-5}{2}}\right).$$
Finally, we get
$$ \frac{\tau''}{\tau}=\frac{(N-1)(N-3)}{4r^2}+o_1\left(r^{-2}\right),\quad r\rightarrow +\infty.$$
Thus $\frac{|\tau''|}{\tau}\leq \frac{C}{r^2}$ for large $r$, which yields, together with the boundness at the origin, the estimate \eqref{boundtau''}. Finally
$$-r^2\frac{d}{dr}\left(\frac{r\tau''}{\tau}\right)-r^2c_0=\frac{(N-1)(N-3)}{4}+o(1),$$ 
As $m>\frac{1}{n-1}$, $N>2$ and thus $\frac{(N-1)(N-3)}{4}>-\frac{1}{4}$. Hence \eqref{tau'''<0}. By Proposition \ref{prop.manif}, the weighted Strichartz estimates \eqref{wstrichartz} hold. 

It remains to check that if $m>1$ (and thus $N>n$), all classical Strichartz estimate hold for $d$-admissible couple with $n<d< N$. By Lemma \ref{locgen} it is sufficient to check Condition (I1):
$$\int_0^\infty \sigma^{\frac{2d}{d-n}}\phi^{n-1}\,dr <+\infty$$
Integrability for small $r$ is ensured by the boundness of $\sigma$. 
By \eqref{polynomialbehaviour}, as $r$ goes to infinity, 
$$ \sigma=\left(\frac{r}{\phi}\right)^{\frac{n-1}{2}}\approx r^{\frac{n-N}{2}},\quad \phi^{n-1}\approx r^{m(n-1)}\approx r^{N-1}$$
Thus 
$$ \sigma^{\frac{2d}{d-n}}\phi^{n-1}\approx r^{-\frac{d(N-n)}{d-n}+N-1}$$
Noting that 
$$-\frac{d(N-n)}{d-n}+N-1=-\frac{n(N-d)}{d-n}-1<-1,$$
we get Condition (I1), which completes the proof of Theorem \ref{theo.poly}.
\end{proof}

Corollary \ref{corol.scattering} is an immediate consequence of Lemma \ref{glgen}.
Let us now give a result concerning global estimates when the volume element grows exponentially at infinity.
\begin{theo}\label{theo.exp'}
Let $M$ be a rotationally symmetric manifold of dimension $n\geq 3$. Assume that there exist $\alpha>0$, such that  
\begin{gather}
\label{strictlynegativecurvature'}
\frac{\tau''}{\tau}\geq \frac{(n-1)^2}{4}\alpha^2-\frac{1/4-\delta_0}{r^2}\\
\label{exponentialbehaviour}
\exists A>0,\quad \phi(r)=e^{\alpha r}(A+o_3(r^{-1})),\quad r\rightarrow +\infty.
\end{gather}
Then the radial solutions of the free equation \eqref{eq:ls} satisfy for all n-admissible couples $(p_j,q_j)$ the weighted Strichartz estimate \eqref{wstrichartz}.
Furthermore for any $d\geq n$, the solutions of \eqref{eq:ls} satisfy all global $d$-admissible Strichartz estimates, and if $$0<p<\frac{4}{n-2}$$
then \eqref{NLS} has short-range behavior.
\end{theo}
It should also be possible to replace \eqref{strictlynegativecurvature'} by a negativity condition on the sectional curvature, however we were not able to write any satisfactory general result in this direction.
\begin{proof}
We will use Proposition \ref{prop.manif} with $c_0=\alpha^2\frac{(n-1)^2}{4}$. 
Assumption \eqref{tau''>0} of Proposition  \ref{prop.manif} is exactly \eqref{strictlynegativecurvature'}.

Let us now check \eqref{boundtau''} and \eqref{tau'''<0}. Let 
$$\psi:=\left(\frac{\phi}{e^{\alpha r}}\right)^{\frac{n-1}{2}}=\frac{\tau}{e^{\frac{n-1}{2}\alpha r}}.$$
By the same calculation as in the proof of Theorem \ref{theo.poly}, \eqref{exponentialbehaviour} implies
\begin{equation}
\label{psi}
\psi=A^{\frac{n-1}{2}}+o_3(r^{-1}).
\end{equation}
Furthermore, $\log \tau=\log \psi+\alpha\frac{n-1}{2}r$. Thus
\begin{equation}
\label{expo1}
\frac{\tau'}{\tau}=\frac{\psi'}{\psi}+\alpha\frac{n-1}{2}.
\end{equation}
Differentiating again, we get
\begin{equation}
\label{expo2}
\frac{\tau''}{\tau}=\left(\frac{\tau'}{\tau}\right)^2+\frac{\psi''}{\psi}-\left(\frac{\psi'}{\psi}\right)^2
=\left(\frac{\psi'}{\psi}+\alpha\frac{n-1}{2}\right)^2+\frac{\psi''}{\psi}-\left(\frac{\psi'}{\psi}\right)^2
\end{equation}
By \eqref{psi}, $\frac{\psi'}{\psi}=o_2(r^{-2})$, $\frac{\psi''}{\psi}=o_1(r^{-3})$ and thus
$$ \frac{\tau''}{\tau}=\alpha^2\frac{(n-1)^2}{4}+o_1(r^{-2}),\quad \frac{d}{dr}\left(r\frac{\tau''}{\tau}\right)=\alpha^2\frac{(n-1)^2}{4}+o(r^{-2}).$$
This yields \eqref{boundtau''} and \eqref{tau'''<0} with $c_0=\alpha^2\frac{(n-1)^2}{4}$. Thus all assumptions of Proposition \ref{prop.manif} holds, which shows that a solution $u$ of \eqref{eq:ls} satisfies all weighted Strichartz estimates \eqref{wstrichartz}. 

To complete the proof of Theorem \ref{theo.exp'}, it remains to check, in view of Lemma \ref{locgen}, that $\phi$ satisfies the condition (I1), which is obvious as $\phi(r)\sim Ae^{\alpha r}$ at infinity. The assertion on the solutions of the nonlinear Scrh\"odinger equations is then an immediate consequence of Lemma \ref{glgen}.
\end{proof}
A simple example of a manifold satisfying the assumptions of Theorem \ref{theo.exp'} is the hyperbolic space, where $\phi(r)=\sinh r$, $sec_r^{rad}=-1$, thus \eqref{strictlynegativecurvature'} and \eqref{exponentialbehaviour} hold with $\alpha=1$ and $A=\frac 12$ (\cite{Ba05}, \cite{Pi05}, \cite{BaCaSt06}). 

We finish with a remark on global estimates when the manifold $M$ is an Euclidian space.

\begin{remark}[Weights in the euclidian radial case]\label{euclidradial}
Let $n\geq 3$. 
The proof of Theorem \ref{theo.poly} is still valid when $\phi(r)=r^m$ with $m$ positive integer. The potential $V$, which is of order $\frac{1}{r^2}$ at the origin satisfy the assumptions of \cite{BuPlStTZ04}. In this case, the Laplacian is exactly the one on $\mathbb{R}^N$ with $N=1+m(n-1)\geq n$, and the volume element is $r^{m(n-1)}=r^{N-1}$. So our $u$ is in fact a radial solution of Schr\"odinger on $\mathbb{R}^N$. We have $\sigma(r)=r^{(1-m)\frac{n-1}{2}}=r^\frac{n-N}{2}$, and so
$$\left\|u\,r^{\frac{N-n}{2}\left(1-\frac{2}{q}\right)}\right\|_{L^p(\mathbb{R},L^q(\mathbb{R}^N))}\leq c\left\|u_0\right\|_{L^2(\mathbb{R}^N)},$$
for all n-admissible couples. In particular we get 
$$\left\|u\,r^\frac{N-n}{n}\right\|_{L^2(\mathbb{R},L^\frac{2n}{n-2}(\mathbb{R}^N))}\leq c\left\|u_0\right\|_{L^2(\mathbb{R}^N)},$$
which represents a gain at infinity combined with a loss at zero.
\end{remark}

\section{Global Strichartz estimates for Schr\"odinger equation with a potential on $\RR^n$}

\label{pot}

In this section we prove Theorem \ref{theo.dispersive}. 

\subsection{Known resolvent estimates with a related potential}
We recall here \cite[Theorem 2.1]{BuPlStTZ04} which is our essential tool in the proof of Theorem \ref{theo.dispersive}.
\begin{reftheo}
\label{theo.BPST}
Let $W\in C^{1}(\RR^n\backslash 0)$, such that 
\begin{gather}
\label{HypW1}
\tag{A1}
|W(x)|\leq \frac{C}{|x|^2}\\
\label{HypW2}
\tag{A2}
\exists\delta_0>0,\quad
\forall x\in \RR^n\setminus \{0\},\quad \left(\frac{n}{2}-1\right)^2+r^2 W\geq \delta_0\\
\tag{A3}
\left(\frac{n}{2}-1\right)^2-r^2\partial_r(r W)\geq \delta_0.
\end{gather}
Then there exists $C>0$ such that
\begin{equation}
\label{BPST}
\sup_{\mu \in \CC\setminus \RR} \Big\||x|^{-1}(-\Delta+W-\mu)^{-1}|x|^{-1}\Big\|_{L^2\rightarrow 
L^2}\leq C.
\end{equation}
\end{reftheo}
The preceding theorem implies weighted $L^2$ estimates on solutions of ($S_W$), which are the main tool to show Strichartz estimates in \cite{BuPlStTZ04}. We will use the same strategy, showing the resolvent estimates in Subsection \ref{sub.resolvent}. Subsections \ref{sub.smooth} and \ref{sub.Str} are devoted to the end of the proof of Theorem \ref{theo.dispersive}. We start with the following lemma.
\begin{lemma}
\label{claimVW}
Let $V\in C^{1}(\RR^N)$ satisfying \eqref{HypBound}, \eqref{HypPositive} and \eqref{HypRepulsive}. Then there exists $W$ satisfying the assumptions of Theorem \ref{theo.BPST} and such that
\begin{equation}
\label{V=W}
\forall x\geq 2R,\quad V(x)=W(x).
\end{equation}
\end{lemma}
\begin{proof}
We choose a nondecreasing radial positive function $\chi$ such that $\chi=1$ for $|x|\geq 2R$, $\chi=0$ for $|x|\leq R$. Let $A>0$ be a large parameter and
$$ W_A:=(1-\chi)\frac{A}{r^2}+\chi V.$$
Note that $W_A\in C^{1}(\RR^n\setminus\{0\})$, and by \eqref{HypBound}, $W_A$ satisfies (A1). Furthermore, if 
$$\left(\frac n2-1\right)^2+A\geq \delta_0,$$
by assumption \eqref{HypPositive}, we get
\begin{equation*}
\left(\frac{n}{2}-1\right)^2+r^2W_A=\left(\left(\frac{n}{2}-1\right)^2+r^2V\right)\chi+\left(\left(\frac{n}{2}-1\right)^2+A\right)(1-\chi)\geq \delta_0 \chi +\delta_0(1-\chi)=\delta_0,
\end{equation*}
By assumption \eqref{HypBound}, $\sup_{x} |x|^2V(x)$ is finite, so we can choose $A$ larger than it. Then, using also \eqref{HypRepulsive},
\begin{equation*}
\left(\frac{n}{2}-1\right)^2-r^2\partial_r(rW_A)=\left(\left(\frac{n}{2}-1\right)^2\chi-r^2\partial_r(rV)\right)+( A-r^2V)r\partial_r \chi+\left(\left(\frac{n}{2}-1\right)^2+A\right)(1-\chi)\geq \delta_0,
\end{equation*}
Thus, if $A$ is large, all assumptions of Theorem \ref{theo.BPST} are satisfied for the potential $W:=W_{A}$. Furthermore, by the definition of $\chi$, $W$ also satisfies \eqref{V=W}.
\end{proof}

\subsection{Resolvent estimates}
\label{sub.resolvent}
Consider the quadratic form $Q(u)=\int |\nabla u|^2+\int V|u|^2$, with domain $D(Q)=H^1$.  By \eqref{HypPositive}, and Hardy's inequality, $Q$ is positive. We define 
$$P_V=-\Delta +V$$ 
to be the self-adjoint operator defined by the Friedrichs extension associated to $Q$. Note that $V$ is bounded and tends to $0$ at infinity, so that the essential spectrum of $P_V$ is $[0,\infty[$ (see e.g. \cite[Theorem XIII.14]{ReSi.T4}). By  \eqref{HypPositive}, $P_V$ does not have any negative eigenvalue. Thus the resolvent $(P_V-\mu)^{-1}$ is well defined for $\mu\in \CC\setminus [0,+\infty)$. Let us show
\begin{prop}
\label{lem.estimates}
There exists $C>0$, such that for all $\lambda\in \RR$, $\forall \eps$, such that $0<|\eps|<1$,
\begin{equation}
\label{RE}
\left\|\brx^{-1}(P_V-\lambda-i\eps)^{-1}\brx^{-1}\right\|_{L^2\rightarrow L^2}\leq \frac{C}{\sqrt{|\lambda|+1}}.
\end{equation}
\end{prop}

\begin{proof}
We must distinguish between the values of $\lambda$. First note that far away from the spectrum $[0,+\infty)$ of $P_V$, estimate \eqref{RE} is obvious:
\begin{lemma}[Elliptic estimates]
\label{claim.ell}
If $\eta_0>0$,  exists $C>0$ such such that \eqref{RE} holds for $\lambda\leq -\eta_0$ and any $\eps$, $0<|\eps|<1$.
\end{lemma}
\noindent
We now turn to the estimates for $\lambda>0$, which are classical.
\begin{lemma}[Estimate for bounded positive $\lambda$]
\label{claim.bounded}
Let $0<\eta_0<M$. There exists $C>0$ (depending only on $M$ and $\eta_0$) such that \eqref{RE} holds for $\lambda\in [\eta_0,M]$, $0<|\eps|<1$.
\end{lemma}
\begin{proof}
The potential $V$ is bounded and $|x|V(x)$ tends to $0$ at infinity, so by Kato's Theorem (\cite[Theorem XII.58]{ReSi.T4}) the spectrum of $P_V$ does not contain any positive eigenvalue. Furthermore, one can write 
 $$V=\frac{1}{\brx^{2}}V_0,$$ 
 with $V_0\in L^{\infty}$, obtaining that $V$ is an Agmon potential. By the Agmon-Kato-Kuroda Theorem \cite[Theorem XIII.33 ]{ReSi.T4}), for any $s>1/2$
$$ \sup_{\substack{\lambda\in [\eta_0,M]\\ 0<|\eps|<1}} \|\brx^{-s}(P_V-\lambda-i\eps)^{-1}\brx^{-s}\|_{L^2\rightarrow L^2}<\infty,$$
which yields estimate \eqref{RE}, with a better weight.
\end{proof}
\begin{lemma}[Estimate for large positive $\lambda$]
\label{claim.large}
There exists $M>0$, $C>0$ such that \eqref{RE} holds for $\lambda\geq M$, $0<|\eps|<1$.
\end{lemma}
\begin{proof}
This is also classical and an immediate consequence of the fact that $V$ is a short-range potential. Recall the estimate on the resolvent of the free operator: for $s>1/2$, there exists $C_s>0$.
\begin{equation}
\label{freeHF}
\forall \lambda\geq 1,\quad \|\brx^{-s}(-\Delta -\lambda -i\eps)^{-1}\brx^{-s} \|\leq \frac{C_s}{\sqrt{\lambda}}.
\end{equation}
Let $f\in C^{\infty}_0(\RR^n)$ and $u:=(P_V-\lambda -i\eps)^{-1}\brx^{-1} f$. Then
$$ -\Delta u-\lambda u-i\eps u=\brx^{-1} f-Vu. $$
Let $s\in (1/2,1)$. By \eqref{freeHF},
$$\|\brx^{-s} u \|_{L^2}\leq \frac{C_s}{\sqrt{\lambda}}\|\brx^{s-1} f\|_{L^2}+\frac{C_s}{\sqrt{\lambda}}\|\brx^{s} V u\|_{L^2}.$$
By \eqref{HypBound}, 
$$|\brx^s V |\leq K\brx^{s-2}\leq K\brx^{-s}$$ 
for some positive constant $K$. Chosing $M$ such that $\frac{K C_s}{\sqrt{M}}$ is strictly less than one, we get for $\lambda\geq M$ a stronger estimate than \eqref{RE}. 
\end{proof}
\noindent
Our last Lemma yields the estimate near $\lambda=0$ which is the only one that derives from the results of \cite{BuPlStTZ04}.
\begin{lemma}
\label{claim0}
There exist $\eta_0,\;C>0$ such that \eqref{RE} holds for $\lambda\in [-\eta_0,+\eta_0]$, $0<|\eps|<1$.
\end{lemma}
\begin{proof}
We divide the proof into two steps. First we shall prove a weaker estimate, and then we shall deduce \eqref{RE}. 

\medskip

\noindent\emph{Step 1: proof of a weaker estimate.}\\
Let us show that for any $\chi\in C_0^{\infty}$, there exists $\eta_0$ and a constant $C$ such that for $\lambda\in [-\eta_0,\eta_0]$, $0<|\eps|<1$,
\begin{equation}
\label{ER.interm}
\left\|\chi(P_V-\lambda-i\eps)^{-1}\brx^{-1}\right\|_{L^2\rightarrow L^2}\leq \frac{C}{\sqrt{\lambda+1}}.
\end{equation}
Note that if \eqref{ER.interm} holds for some $\chi\in C_0^{\infty}(\RR^d)$, then it holds for any $\widetilde{\chi}\in C_0^{\infty}(\RR^d)$ with $\supp \widetilde{\chi}\subset \{\chi\geq 1\}$. Indeed,
$$\left\|\widetilde{\chi}(P_V-\lambda-i\eps)^{-1}\brx^{-1}\right\|_{L^2\rightarrow L^2}\leq \left\|\chi (P_V-\lambda-i\eps)^{-1}\brx^{-1}\right\|_{L^2\rightarrow L^2}$$
Let $R>0$ arbitrary. It is sufficient to prove \eqref{ER.interm} for $\chi\in C_0^{\infty}(\RR^d)$ satisfying
\begin{equation}
\label{chi1}
|x|\leq 3R\Longrightarrow \chi(x)=1,
\end{equation}
where $R$ is given by Lemma \ref{claimVW}. We argue by contradiction. If \eqref{ER.interm} does not hold, there exist sequences $\lambda_n, \,\eps_n\in \RR$ , $f_n,\,u_n\in L^2$  such that
\begin{gather}
\label{lambda.eps}
0<\eps_n<1,\quad \lim_{n\rightarrow +\infty} \lambda_n=0\\
\label{CVabsurd}
\|\chi u_n\|_{L^2}=1,\quad \lim_{n\rightarrow +\infty} \|f_n\|_{L^2}=0\\
\label{eqabsurd}
(-\Delta +V-\lambda_n-i\eps_n)u_n=\brx^{-1} f_n.
\end{gather}
Let us first show
\begin{gather}
\label{boundbr}
\exists C>0,\; \forall n,\;\|\brx^{-1}u_n\|_{L^2}\leq C,\\
\label{boundL2loc}
\forall \tilde{\chi}\in C^{\infty}_0,\; \exists C>0,\; \forall n, \quad\|\tilde{\chi} u_n\|_{L^2}+ \|\tilde{\chi}\nabla u_n\|_{L^2}\leq C.
\end{gather}
Let $ \tilde{\chi}\in C^{\infty}_0$. A straightforward integration by parts gives us from \eqref{eqabsurd}
\begin{equation}\label{IBP}
\int \tilde{\chi}|\nabla u_n|^2=\int\Delta\tilde{\chi}\,|u_n|^2+\re \tilde{\chi}\brx^{-1}f_n\overline{u}_n-\tilde{\chi}(V-\lambda_n)|u_n|^2.
\end{equation}
So by \eqref{chi1} and \eqref{CVabsurd}, estimate \eqref{boundL2loc} holds if $\supp \tilde{\chi}\subset\{|x|\leq 3R\}$.
Let $\psi\in C^{\infty}_0$ such that
\begin{equation}
\label{goodpsi}
|x|\leq 2R\Longrightarrow \psi(x)=0 ,\quad |x|\geq 3R\Longrightarrow \psi(x)=1.
\end{equation}
Then
\begin{equation*}
(-\Delta +V-\lambda_n-i\eps_n)\psi u_n=\brx^{-1} \psi f_n-\Delta \psi u_n-2\nabla \psi \cdot\nabla u_n.
\end{equation*}
As $\supp\psi \subset \{|x|\geq 2R\}$, we may replace $V$ in the preceding equation by the potential $W$ given by Lemma \ref{claimVW}. 
Thus by Theorem \ref{theo.BPST}, 
$$\|\brx^{-1}\psi u_n\|_{L^2}\leq C\left\|f_n-\brx\Delta \psi u_n-2\brx\nabla \psi \cdot\nabla u_n\right\|_{L^2}.$$ 
Since $\supp \nabla \psi\subset \{|x|\leq 3R\}$, we can use \eqref{boundL2loc}, and get an uppper-bound for the right hand side, independent of $n$. This, together with the first part of \eqref{CVabsurd}, gives \eqref{boundbr}. The first part of estimate \eqref{boundL2loc} is a direct consequence of \eqref{boundbr}. The second part follows from \eqref{IBP}. \\

According to \eqref{boundL2loc}, $u_n$ is bounded in $H^1_{\loc}$. Extracting subsequences if necessary, there exist $u\in L^2_{\loc}(\RR^n)$ such that
\begin{equation}
\label{limite}
\lim_{n\rightarrow \infty} u_n =u \text{ in }L^2_{\loc}.
\end{equation}
According to \eqref{CVabsurd} and \eqref{eqabsurd}, $u$ satisfies the equation
\begin{equation}
\label{eqresonnance}
-\Delta u+V u+i\eps u=0
\end{equation}
Furthermore, by \eqref{boundbr}
\begin{equation}
\label{bruL2}
\brx^{-1} u\in L^2.
\end{equation}
Let us show that $u=0$, which will contradict, together with \eqref{limite}, the equality $\|\chi u_n\|=1$ in \eqref{CVabsurd}. Let $\varphi \in C^{\infty}_0$ such that $\varphi(x)=1$ if $|x|\leq 1$ and $\varphi=0$ if $|x|\geq 2$. Multiplying \eqref{eqresonnance} by $\overline{u}\varphi(x/\rho)$ and taking the real part, we get
\begin{equation}
\label{IPPrho}
\int |\nabla u|^2\varphi(x/\rho)+\int V|u|^2\varphi(x/\rho)-\int \frac{1}{\rho^2} \Delta \varphi(x/\rho)|u|^2=0.
\end{equation}
Furthermore, 
$$\int \frac{1}{\rho^2} \Delta \varphi(x/\rho)|u|^2\leq C\int_{\rho\leq|x|\leq 2\rho} \frac{1}{|x|^2}|u|^2,$$
which tends to $0$ as $\rho$ tends to infinity in view of \eqref{bruL2}. Noting that by \eqref{HypBound} and \eqref{bruL2}, $\int V|u|^2$ is finite, so lletting $\rho$ tends to infinity in \eqref{IPPrho}, we get that $\nabla u\in L^2$ and
\begin{equation}
\label{positivity}
\int |\nabla u|^2+\int V|u|^2=0.
\end{equation}
This shows that $u=0$ by assumption \eqref{HypPositive}, and Hardy's inequality, which imply that  
$$ \forall u\in \hdot,\quad \int |\nabla u|^2+\int V|u|^2\geq \delta_0\int|\nabla u|^2.$$
(see e.g. Proposition 1.3 of \cite{BuPlStTZ04}). The proof of \eqref{ER.interm} is complete.

\medskip

\noindent\emph{Step 2: end of the proof.}\\
Take $\lambda\in [-\eta_0,\eta_0]$ and $0<\eps <1$. If $f\in C^{\infty}_0(\RR^n)$ and $u=(P_V-\lambda-i\eps)^{-1} (\brx^{-1} f)$, we have
$$ (-\Delta +V -\lambda-i\eps)u =\brx^{-1}f.$$
According to \eqref{ER.interm}, for any $\chi \in C^{\infty}_0(\RR^n)$, there exists a constant $C>0$ such that
\begin{equation}
\label{chiu}
\|\chi u\|_{L^2}+\|\chi \nabla u\|_{L^2}\leq C\|f\|_{L^2}.
\end{equation}
Let $\psi$ be as in \eqref{goodpsi}. Then
$$ (-\Delta +W-\lambda-i\eps)(\psi u)=\brx^{-1}\psi f-\Delta \psi u -2\nabla\psi\cdot\nabla u.$$
Hence by Theorem \ref{theo.BPST}
$$ \|\brx^{-1}\psi u\|_{L^2}\leq \|\psi f-\brx \Delta \psi u -2\brx \nabla\psi\cdot\nabla u\|_{L^2}$$
which yields, together with \eqref{chiu}, the inequality
$$ \|\brx^{-1} u\|_{L^2}\leq C\|f\|_{L^2}.$$
The proof of Lemma \ref{claim0} is complete.
\end{proof}
\noindent
Putting together Lemmas \ref{claim.ell}, \ref{claim.bounded}, \ref{claim.large} and \ref{claim0}, we get \eqref{RE} for all $\lambda \in \RR$ and all $\eps\neq 0$, which concludes the proof of Proposition \ref{lem.estimates}.
\end{proof}
\begin{remark}
In the proof of the last Lemma, we have shown that $P_V$ does not admit any $0$ resonnance.
\end{remark}

The end of the proof of Theorem \ref{theo.dispersive} follows the strategy of \cite{BuPlStTZ04}. We recall it for the  sake of completness in the following two subsections. 

\subsection{Proof of smoothing effect}
\label{sub.smooth}

From standard arguments that go back to \cite{Ka65} (see \cite[Proposition 2.7]{BuGeTz04} for an elementary proof), it is sufficient to show that for $\eps\neq 0$,  $\brx^{-1}(P_V-\lambda-i\eps)^{-1} \brx^{-1}$ extends to a map from $H^{-\frac 12}$ to $H^{\frac 12}$ with the following uniform bound:
\begin{equation}
\label{goodIR}
\exists C>0,\; \forall \lambda\in \RR,\; \forall \eps\neq 0,\quad \left\|\brx^{-1}(P_V-\lambda-i\eps)^{-1} \brx^{-1}\right\|_{H^{-1/2}\rightarrow H^{1/2}}\leq C.
\end{equation}

We first show
\begin{equation}
\label{IR2}
\exists C>0,\; \forall \lambda\in \RR,\; \forall \eps\neq 0,\quad \left\|\brx^{-1} (P_V-\lambda-i\eps)^{-1} \brx^{-1}\right\|_{L^2\rightarrow H^{1}}\leq C.
\end{equation}
Let $f\in L^2$ and $u=(P_V-\lambda-i\eps)(\brx^{-1} f)$. Then by Lemma \ref{lem.estimates}
\begin{equation}
\label{est.L2}
\left\|\brx^{-1} u\right\|_{L^2}\leq \frac{C}{\sqrt{|\lambda|+1}}\|f\|_{L^2}.
\end{equation}
To get informations on the gradient of $\brx^{-1} u$, we consider
$$ (P_V-\lambda-i\eps)\big(\brx^{-1}u\big)=\underbrace{-\Big(\Delta\brx^{-1}\Big)u-2\nabla \brx^{-1}\cdot \nabla u -(\lambda+i\epsilon)\brx^{-1} u}_{F}.$$
Multiplying by $\brx^{-1}\ubar$, integrating on $\RR^d$ and taking the real part, we obtain
\begin{equation}
\label{est.L2bis}
\left\|\nabla \big(\brx^{-1}u\big)\right\|^2_{L^2}\leq C\left(\left\|\brx^{-1}V u\right\|^2_{L^2}+\lambda \left\|\brx^{-1} u\right\|^2_{L^2}+\left|\re \int \brx^{-1}F \ubar\right|\right).
\end{equation}
It remains to bound $\re \int \brx^{-1}F \ubar$. By integration by part:
$$ 2\re \int \nabla \brx^{-1}\cdot \nabla u \brx^{-1}\ubar=\int \brx^{-1}\nabla \brx^{-1}\cdot \nabla |u|^2=-\int \divergence (\brx^{-1}\nabla \brx^{-1})|u|^2, $$
From Cauchy-Schwarz inequality and the bound $|D^k\brx^{-1}|\leq C_k\brx^{-(k+1)}$ (where $D^k$ is any derivative of order $k$), we get
$$\left|\re \int \brx^{-1}F \ubar\right|\leq C(1+|\lambda|)\left\|\brx^{-1}u\right\|^2_{L^2}.$$
Together with \eqref{est.L2}, \eqref{est.L2bis}  and the boundedness of $V$, we get
$$ \left\| \brx^{-1}u\right\|_{H^1}\leq C\|f\|_{L^2},$$
hence \eqref{IR2}.

The end of the proof of \eqref{goodIR} is now very classical. Noting that the adjoint of the bounded operator on $L^2$ $\brx^{-1}(P_V-\lambda-i\eps)^{-1} \brx^{-1}$ is $\brx^{-1}(P_V-\lambda+i\eps)^{-1} \brx^{-1}$, and using \eqref{IR2} with $-\eps$ instead of $\eps$ we get that $\brx^{-1}(P_V-\lambda-i\eps)^{-1} \brx^{-1}$ extends to a map from $H^{-1}$ to $L^2$ with the bound
\begin{equation}
\label{IR3}
\left\|\brx^{-1}(P_V-\lambda-i\eps)^{-1} \brx^{-1}\right\|_{H^{-1}\rightarrow L^2}\leq C.
\end{equation}
Interpolating between \eqref{IR2} and \eqref{IR3}, we get \eqref{goodIR}, which concludes the proof of the smoothing effect in Theorem \ref{theo.dispersive}.

\subsection{Proof of Strichartz estimates}\label{sub.Str}
We shall need the Lorentz Spaces $L^{\frac{2n}{n-2},2}$, $L^{\frac{2n}{n+2},2}$ and $L^{n,\infty}$. Recall that $L^{\frac{2n}{n-2},2}$ is slightly smaller than $L^{\frac{2n}{n-2}}$, that $L^{\frac{2n}{n+2},2}$ is the dual of $L^{\frac{2n}{n-2},2}$, and that a smooth function of order $\frac{1}{|x|}$ at infinity is in $L^{n,\infty}$, commonly refered as weak $L^n$. O'Neil inequality states a generalization of H\"older inequality for Lorentz spaces \cite{ON63},
\begin{equation}
\label{ONeil}
\|FG\|_{L^{\frac{2n}{n+2},2}}\leq \|F\|_{L^{2}}\|G\|_{L^{n,\infty}}.
\end{equation}
Furthermore, by the refined endpoint inequality (see \cite{KeTa98}), there is a constant $C>0$ such that if $U$ is a solution of the free Schr\"odinger equation on $\RR^n$
$$\left\{\begin{array}{c}
 i\partial_t U+\Delta U=F,\\
U(0)=U_0,
\end{array}\right.$$
then
\begin{equation}
\label{freeendpoint}
\|U\|_{L^2(\RR,L^{\frac{2n}{n-2},2})}\leq C\Big(\|U_0\|_{L^2}+\|F\|_{L^2(\RR,L^{\frac{2n}{n+2},2})}\Big).
\end{equation}

We shall adapt the argument of \cite[section 3]{BuPlStTZ04}, to the case where the right-member $f$ of \eqref{LS} is nonzero. In \cite{BuPlStTZ04}, this case is not considered ; it has been recently treated in \cite{Pi05}. For the sake of completeness, we give here the proof. We first show the endpoint estimate \eqref{Strichartz} with $p_1=p_2=2$ and $q_1=q_2=\frac{2n}{n-2}$
\begin{equation}
\label{Strichartz0}
\|u\|_{L^2(\RR,L^{\frac{2n}{n-2}})}\leq C\left(\|u_0\|_{L^2}+\|f\|_{L^2(\RR,L^{\frac{2n}{n+2}})}\right).
\end{equation}
Writing 
$$\left\{\begin{array}{c}
i\partial_t u+\Delta u=f+V u,\\
 u(0)=u_0,
 \end{array}\right.$$
and using \eqref{freeendpoint}, we get
\begin{equation}\label{Strichartz1}
 \|u\|_{L^2,L^{\frac{2n}{n-2},2}}\leq C\left(\|u_0\|_{L^2}+\|Vu\|_{L^2(\RR,L^{\frac{2n}{n+2},2})}+\| f\|_{L^2(\RR,L^{\frac{2n}{n+2},2})}\right).
 \end{equation}
One one hand, by \eqref{ONeil}, assumption \eqref{HypBound} on $V$, and the smoothing estimate shown in the previous subsection,
\begin{equation*}
\|V u\|_{L^2(\RR,L^{\frac{2n}{n+2},2})}\leq\|\brx V\|_{L^{n,\infty}}\|\brx^{-1}u\|_{L^2(\RR,L^{2})}\leq C\left(\|u_0\|_{L^2}+\|\brx f\|_{L^2(\RR,L^2)}\right),
\end{equation*}
On the other hand, again by \eqref{ONeil},
$$\|f\|_{L^2(\RR,L^{\frac{2n}{n+2},2})}\leq \|\brx^{-1} \|_{L^{n,\infty}}\|\brx f\|_{L^2(\RR,L^2)}.$$
Hence there is a constant $C>0$ such that for any solution of \eqref{LS} we have
\begin{equation}
\label{Strichartz2}
\|u\|_{L^2(\RR,L^{\frac{2n}{n-2},2})}\leq C\left(\|u_0\|_{L^2}+\|\brx f\|_{L^2(\RR,L^2)}\right).
\end{equation}
In particular, we have obtained endpoint Strichartz estimate for the homogeneous equation
\begin{equation}\label{hom}
\|e^{itP_V}u_0\|_{L^2(\RR,L^{\frac{2n}{n-2},2})}\leq C\|u_0\|_{L^2},
\end{equation}
and for the inhomogeneous equation with zero initial data the weighted estimate
\begin{equation}
\label{inhom}
\|A(f)\|_{L^2(\RR,L^{\frac{2n}{n-2},2})}\leq C\|\brx f\|_{L^2(\RR,L^2)},
\end{equation}
where we denote
$$A(f)(t,x)=i\int_0^te^{i(t-\tau)P_V}f(\tau,x)\,d\tau.$$
We are left with proving the endpoint Strichartz estimate for $A(f)$. We shall do it by duality. 
Let $g\in \mathcal{C}^\infty_0(\mathbb{R}\times\mathbb{R}^n)$. We choose $T>0$ such that $\supp g\subset (-T,T)\times \mathbb{R}^n$. It follows that for positive $t$,
$$A^*(g)(t,x)=i\int_t^T e^{i(t-\tau)P_V}g(\tau,x)\,d\tau,$$
and so $A^*(g)$ is a solution of the backwards inhomogeneous equation, with source term $g$ and zero initial data at time $T$.
Then, from \eqref{Strichartz2}, 
$$\|A^*(g)\|_{L^2((0,T),L^{\frac{2n}{n-2},2})}\leq C \|\brx g\|_{L^2(\RR,L^2)},$$
where $C$ is independent of $T$, and similarly for negative time, so that
$$\|A^*(g)\|_{L^2(\mathbb{R},L^{\frac{2n}{n-2},2})}\leq C \|\brx g\|_{L^2(\RR,L^2)}.$$
The constant $C$ is independent of $g$, so it follows by duality that $A(f)$ is in the dual of $L^2(\mathbb{R},L^2(\brx dx))$, with the norm estimate
\begin{equation}
\label{Strichartz3}
\|\brx^{-1}A(f)\|_{L^2(\RR,L^2)}\leq C\|f\|_{L^2(\RR,L^{\frac{2n}{n+2},2})}.
\end{equation}
By \eqref{ONeil},
$$\|VA(f)\|_{L^2(\RR,L^{\frac{2n}{n+2},2})}\leq \|\brx V\|_{L^{n,\infty}}\|\brx^{-1} A(f)\|_{L^2(\RR,L^{2})}\leq C\|f\|_{L^2(\RR,L^{\frac{2n}{n+2},2})}.$$
Therefore \eqref{Strichartz1} gives us the endpoint Strichartz estimates for the zero-initial data inhomogeneous problem,
\begin{equation}
\label{u_2}
\|A(f)\|_{L^2(\RR,L^{\frac{2n}{n-2},2})}\leq C\|f\|_{L^2(\RR,L^{\frac{2n}{n+2},2})}.
\end{equation}
Summing with \eqref{hom} we obtain
\begin{equation}
\label{refined.endpoint}
\|u\|_{L^2(\RR,L^{\frac{2n}{n-2},2})} \leq C\left(\|u_0\|_{L^2}+\|f\|_{L^2(\RR,L^{\frac{2n}{n+2},2})}\right).
\end{equation}
In conclusion, the endpoint Strichartz estimate \eqref{Strichartz0} holds. 

Now, writing $\frac{d}{dt}\int |u|^2=\im \int f \overline{u}$, we get by H\"older inequality that for any admissible couple $(p,q)$
$$ \|u(t)\|_{L^2}^2-\|u_0\|^2_{L^2}\leq \frac{1}{2}\|f\|^2_{L^{p'}(\RR,L^{p'})}+\frac{1}{2}\|u\|_{L^{p}(\RR,L^{q})}^2.$$
Taking $p=\infty$, $q=2$ and $p=2$, $q=\frac{2n}{n-2}$ yields (using also \eqref{refined.endpoint} for the second line)
\begin{align}
\label{Strichartza}
\|u\|_{L^{\infty}(\RR,L^2)}&\leq C\Big(\|u_0\|_{L^2}+\|f\|_{L^1(\RR,L^2)}\Big)\\
\label{Strichartzb}
\|u\|_{L^{\infty}(\RR,L^2)}&\leq C\Big(\|u_0\|_{L^2}+\|f\|_{L^2(\RR,L^\frac{2n}{n+2})}\Big).\\
\intertext{By the same duality argument than above, we can deduce from \eqref{Strichartzb}}
\label{Strichartzc}
\|u\|_{L^{2}(\RR,L^{\frac{2n}{n-2}})}&\leq C\Big(\|u_0\|_{L^2}+\|f\|_{L^1(\RR,L^2)}\Big).
\end{align}

Estimates for other values of $p_1$, $p_2$, $q_1$ and $q_2$ follow from interpolation between \eqref{Strichartz0}, \eqref{Strichartza}, \eqref{Strichartzb} and \eqref{Strichartzc}.

Note that we only need \eqref{smoothing} with $L^2$ instead of $H^{1/2}$ to show the Strichartz estimates. However, we stated the smoothing property \eqref{smoothing} for its own interest.

\bibliographystyle{acm} 
\bibliography{StrichartzManifolds}

\bigskip

\end{document}